\documentclass[a4paper,12pt]{amsart}
\usepackage[cp1251]{inputenc}

\usepackage[T2A]{fontenc}
\usepackage[bulgarian, english]{babel}

\usepackage{amsmath,amssymb}

\newtheorem{theorem}{Theorem}
\newtheorem{lemma}[theorem]{Lemma}

\newtheorem{proposition}[theorem]{Proposition}
\newtheorem{corollary}[theorem]{Corollary}
\newtheorem{definition}[theorem]{Definition}
\theoremstyle{remark}\newtheorem{remark}{Remark}
\theoremstyle{remark}
\newcommand{\tuple}[1]{\left<{#1}\right>}

\title{Some inequalities for norms in $\mathbb{ R}^n$ and $\mathbb{ C}^n$}

\author{Stefan Gerdjikov}

\address{S. Gerdjikov\\Faculty of Mathematics and Informatics\\ Sofia University\\James Bourchier Blvd. 5, 1164 Sofia, Bulgaria \\
\vspace{1mm}
\newline
Institute for Information and Communication Technologies\\
\\Bulgarian Academy of Sciences\\Acad. G. Bonchev 2, 1113 Sofia, Bulgaria
}

\email{}

\author{Nikolai Nikolov}

\address{N. Nikolov\\Institute of Mathematics and Informatics\\Bulgarian Academy
of Sciences\\Acad. G. Bonchev 8, 1113 Sofia, Bulgaria\vspace{1mm}
\newline Faculty of Information Sciences\\
State University of Library Studies and Information Technologies\\Shipchenski prohod 69A,
1574 Sofia,	Bulgaria}

\email{nik@math.bas.bg}

\thanks{The second named author was partially supported by the Bulgarian National Science Fund,
Ministry of Education and Science of Bulgaria under contract KP-06-N52/3.}

\subjclass[2010]{52A40, 52A21, 32F17}

\begin{document}
	
\keywords{norm, convex domain}
	
\begin{abstract} {}
The main result of this paper is that for any norm on a complex or real $n$-dimensional linear space, every extremal basis satisfies inverted triangle inequality with scaling factor $2^n-1$. Furthermore, the constant $2^n-1$ is tight. We also prove that the norms of any two extremal bases are comparable with a factor of $2^n-1$, which, intuitively, means that any two extremal bases are quantitatively equivalent with the stated tolerance.
\end{abstract}
	
\maketitle
	
\section{Introduction}
\emph{Extremal bases}, originally introduced in~\cite{M92,H02}, have been established as a useful tool in the study of the properties of the so called $\mathbb{C}$-convex domains, $D$. On the one hand side they induce a natural orthonormal coordinate system around any given point $z$ in the interior of the domain $D$. On the other hand, under certain assumptions, see below, every extremal basis $B$ at point $z$ enjoys the following inequalities:
\begin{equation}\label{eq_basic}
\left(\sum_{b\in B} \frac{|\tuple{b,v}|}{d_D(z;b)}\right)^{-1}\lesssim d_D(z;v) \lesssim \left(\sum_{b\in B} \frac{|\tuple{b,v}|}{d_D(z;b)}\right)^{-1},
\end{equation}
where $d_D(z;v)$ corresponds to the distance from $z$ to the boundary of $D$ in direction $v$, i.e.:
\begin{equation*}
d_D(z;v) = \sup \{r\in \mathbb{R}^+|\ z+\lambda v\in D\text{ whenever }|\lambda| < r\}.
\end{equation*}
This means that the extremal bases provide a convenient linear approximation of the structure of the body $D$ in a neighbourhood of any given point $z$ in the interior of the body.

The property~\ref{eq_basic} of extremal bases has facilitated the construction of plurisubharmonic functions with bounded Hessians and for obtaining of estimates for the Bergman kernels,~\cite{M92,M94}. In~\cite{H02,H04}, the author states and uses the estimates~\ref{eq_basic} to obtain H\"older and $L^p$ estimates for the solutions of Cauchy-Riemann equations on smooth bounded pseudoconvex domains $D$ of finite type. Estimates~\ref{eq_basic} have been also applied in the study of the Kobayashi and Bergman metrics,~\cite{M92,NP03,NPZ11,NPT13} and more recently,~\cite{W22}. For a survey on the geometry of extremal bases and their applications we refer to~\cite{CD14}.

The construction of maximal bases can be described as a greedy procedure. One starts with an arbitrary point $z$ in the interior of a $\mathbb{C}$-convex domain $D$ and an empty set of vectors/directions $B_0$. Next, inductively, the current set $B_k$ is extended to $B_{k+1}$ by adding an \emph{extremal} direction $v_{k+1}$ that is orthogonal to the subspace spanned by $B_k$. The process terminates once the set $B:=B_n$ spans the entire space in which case $B$ is the desired basis.

The notion of \emph{extremity} can be specified as \emph{maximal},~\cite{M92,M94}, or as \emph{minimal},~\cite{H02,H04}, and the difference consists in whether one selects:
\begin{eqnarray*}
v_{k+1} &\in & \arg\max \{d_{D}(z;v)\,|\, v\perp span(B_k)\} \text{ or } \\
v_{k+1} &\in &\arg\min \{d_{D}(z;v)\,|\, v\perp span(B_k)\}.
\end{eqnarray*}

Though important for the applications, the proofs of the estimates~\ref{eq_basic} departing from geometric or analytical view points, depend on some kind of smoothness conditions for the domain $D$,~\cite{M92,M94,NPT13}, and often provide only implicit or rough estimates for the hidden constant. In particular, it is not evident if and how it depends on the domain $D$.

In this paper, we give answer to the above questions by proving that the estimates~\ref{eq_basic} are valid with constant $2^{n}-1$ where $n$ is the dimensionality of the space in which the (so called weakly linear) convex domain $D$ resides. In this sense this constant is independent of $D$. Furthermore it turns out that the constant $2^n-1$ is sharp, that is it cannot be improved.

Our approach is algebraic. It departs from the observation that for a weakly linear convex domain $D$ and any particular point $z$ in the interior of $D$, the function:
\begin{equation*}
f(v)=\frac{1}{d_D(z;v)}
\end{equation*}
is a semi-norm. If further the domain $D$ does not contain complex lines, then $f$ is a norm, but see Remark~\ref{seminorm} which suggests that this assumption is not essential and it is only for technical reasons that we restrict our considerations to norms. Thus the estimates~\ref{eq_basic} can be restated as:
\begin{equation}\label{eq_basic2}
\sum_{b\in B} |\tuple{b,v}|f(b) \gtrsim f(v) \gtrsim \sum_{b\in B} |\tuple{b,v}| f(b)
\end{equation}
and wheres the first inequality is satisfied with constant $1$, as it easily follows by the triangle inequality, the second inequality seems to be more
challenging.

To keep the outline self-contained, in Section~\ref{sec:preliminaries} we prove that every (bounded on the unit sphere) norm can be represented as:
\begin{equation*}
f(v)=\sup_{u\in U} |\tuple{v,u}|
\end{equation*}
for an appropriate set of vectors $U$. The geometric interpretation of $U$ is that the vectors $u\in U$ define the supporting hyperplanes to the body $D$ centred at $z$.

In Section~\ref{sec:min_max} we state and prove that for any extremal basis $B$:
\begin{equation*}
(2^n-1) f(v) \ge \sum_{b\in B} |\tuple{b,v}| f(b).
\end{equation*}
We should stress that the definition of maximal and minimal in our notation, see Definition~\ref{def:min_max}, are reciprocal to the notions used for bodies. The reason is that maximising $d_D(z;.)$ is equivalent to minimising $f$ and vice versa. Thus, Theorem~\ref{minimal} proves the statement for maximal bases w.r.t. $d_D(z;.)$ and for minimal bases w.r.t. $f$, whereas Theorem~\ref{maximal} proves the statement for minimal bases w.r.t. $d_D(z;.)$ and maximal bases w.r.t. $f$.

In Section~\ref{sec:min_max} we prove that $2^n-1$ is the best possible bound. Namely, we show that for every $\varepsilon>0$ there are norms whose maximal, resp. minimal, bases violate the inequality:
\begin{equation*}
(2^n-1-\varepsilon) f(v) \ge \sum_{b\in B} |\tuple{b,v}| f(b)
\end{equation*}
for at least one vector $v$. Since every norm gives rise to a convex body $D'=\{v\,|\, f(v)\le 1\}$ such that $d_{D'}(0;v)=\frac{1}{f(v)}$, the results translate immediately for convex bodies. Again, Proposition~\ref{min_asymptotic} handles the case of maximal basis w.r.t. $d_D(z;.)$ and minimal basis w.r.t. $f$, whereas Proposition~\ref{max_asymptotic} handles the case of minimal basis w.r.t. $d_D(z;.)$ and maximal basis w.r.t. $f$.

In Section~\ref{sec:comparison} we show that minimal and maximal bases are equivalent in their norms. Whereas similar result has been previously proven in~\cite{NPT13}, the bounds in Section~\ref{sec:comparison} are based on more accurate analysis of the algebraic structure and improve the estimate fo the constants from $2^nn!$ to $2^n$, thus reducing a factor of $n!$.

We conclude in Section~\ref{sec:open} with some open problems.

\section{Preliminaries}\label{sec:preliminaries}
In what follows we assume that $V$ is a linear space on $\mathbb{F}=\mathbb{R}$ or $\mathbb{F}=\mathbb{C}$ and that $V$ is supplied with a scalar product $\tuple{\cdot,\cdot}:V\times V\rightarrow \mathbb{F}$. We denote with $\|.\|:V\rightarrow \mathbb{R}^+$ and $\mathbb{S}_1$ the induced norm and the unit sphere in $V$, respectively, i.e.:
\begin{eqnarray*}
\|u\| &=& \tuple{u,u} \text{ for } u\in V\\
\mathbb{S}_1 &=& \{u\in V\,|\, \tuple{u,u}=1\}.
\end{eqnarray*}

\begin{lemma}\label{basic}
Let $f:V\rightarrow \mathbb{R}^+$ be a norm. Assume that $f$ is bounded on the unit sphere $\mathbb{S}_1\subset V$.
\begin{enumerate}
\item There is a set of vectors $U\subseteq V$ such that:
\begin{equation*}
f(v)=\sup_{u\in U} |\tuple{v,u}|.
\end{equation*}
\item If $v_0\in V$ is a unit vector such that $f(v_0)=\sup_{v\in \mathbb{S}_1} f(v)$, then for all $u\in V$ it holds that:
\begin{equation*}
f(u) \ge |\tuple{v,u}|f(v).
\end{equation*}
\end{enumerate}
\end{lemma}
\begin{proof}
\begin{enumerate}
\item The first part is an immediate consequence of Hahn-Banach Theorem. Indeed, let $v\in V$. Then denoting by $V_0=span(v)$, we define $L_0:V_0\rightarrow \mathbb{F}$ as
$L_0(\alpha v)=\alpha f(v)$. Clearly, $L_0$ is a linear functional on $V_0$ and $|L_0(\alpha v)| =|\alpha|f(v)=f(\alpha v)$. By Hanh-Banach Theorem, since $f$ is a norm, we can extend $L_0$ to a linear functional $L_v:V\rightarrow \mathbb{F}$ such that $|L_v(u)|\le f(u)$ for all $u\in V$. Since $L_v$ is linear on $V$ and $\sup_{u\in \mathbb{S}_1} L_v(u)$ is bounded above by $\sup_{u\in \mathbb{S}_1} f(u)<\infty$, it follows that $L_v$ is a bounded linear operator and cosequently there is a vector $v^*\in V$ such that:
\begin{equation*}
\tuple{u,v^*}=L_v(u) \text{ for all } u\in V.
\end{equation*}
Let $U=\{v^*\,|\, v\in V\}$. It is straightforward that for any $u\in V$ and $v^*\in U$, $|\tuple{u,v^*}|=|L_v(u)|\le f(u)$. On the other hand $f(u)=L_u(u)=|\tuple{u,u^*}|$. Therefore:
\begin{equation*}
f(v)=\sup_{u\in U} |\tuple{v,u}|.
\end{equation*}

\item For the second part, consider the linear functional $L=L_{v_0}$ induced by $v_0$. Since $L$ is linear, it is determined by its values on an orthonormal basis on $V$. Without loss of generality we may and we do assume that $(e_i)_{i\in I}$ is such a basis with $e_1=v_0$. Assume that $L(e_i)\neq 0$ for some $i>1$. Then we consider the vector:
\begin{equation*}
u= \overline{L(e_1)} v_0 + \overline{L(e_i)} e_i=f(v_0) v_0 + \overline{L(e_i)} e_i.
\end{equation*}
Then we have that $\|u\|=\sqrt{f^2(v_0) + |L(e_i)|^2}$ and:
\begin{equation*}
L(u) = f(v_0) L(v_0)+\overline{L(e_i)}L(e_i)=f(v_0)^2 + |L(e_i)|^2.
\end{equation*}
Hence $u'=\frac{u}{\|u\|}$ is a unit vector and $L(u')= \sqrt{f(v_0)^2 + |L(e_i)|^2}$ implying that $L(u')> f(v_0)$ since $L(e_i)\neq 0$. However, $f(u')\ge L(u')>f(v_0)$ and this contradicts the maximality of $v_0$. Consequently, $L(e_i)=0$ for any $i\neq 1$. This proves that $L(u)=\tuple{u,e_1}L(e_1)=\tuple{u,v_0} f(v_0)$ and therefore:
\begin{equation*}
|\tuple{u,v_0}|f(v_0)=|L(u) | \le f(u)
\end{equation*}
as required.
\end{enumerate}
\end{proof}

\begin{remark}\label{seminorm}
We can view the first part of the above lemma as a characterisation of (semi)norms.
Indeed, if $U\subseteq V$, then $f_U:V\rightarrow \mathbb{R}^+$ defined as:
\begin{equation*}
f_U(v)=\sup_{u\in U} |\tuple{v,u}|
\end{equation*}
is a semi-norm. Actually, it is a norm iff $U^{\perp}=\{v\in V\,|\, \forall u\in U(\tuple{u,v}=0)\}$ is the trivial set $\{0\}$. Since, $U^{\perp}$ is a subspace of $V$, the properties of $f_U$ are uniquely determined by the behaviour of $f_U$ on the orthogonal subspace of $U^{\perp}$. Indeed, if $v=v'+u$ with $u\in U^{\perp}$ and $v'\perp U^{\perp}$ then:
\begin{equation*}
f_U(v') \le f_U(v) + f_U(u)=f(v) \text{ and } f_U(v)\le f_U(v') + f_U(-u)=f_U(v'),
\end{equation*}
that is $f_U(v')=f_U(v)$.
\end{remark}

\section{Minimal and Maximal Bases}\label{sec:min_max}
In this section we assume that $n$ is a positive integer, and $V$ is an $n$-dimensional linear space supplied with a scalar product $\tuple{\cdot,\cdot}:V\times V\rightarrow \mathbb{F}$ where $\mathbb{F}\in \{\mathbb{R},\mathbb{C}\}$.
\begin{definition}\label{def:min_max}
Let $f:V\rightarrow \mathbb{R}^+$ be a norm. We define an \emph{$f$-minimal ($f$-maximal, resp.)} (orthonormal) basis for $V$ inductively as follows:
\begin{itemize}
\item $n=1$, then for any $b\in \mathbb{S}_1$, $(b)$ is an $f$-minimal ($f$-maximal, resp.) basis.
\item $n>1$, then let:
\begin{eqnarray*}
b_1&\in& \arg\min \{f(b)\,|\, b\in \mathbb{S}_1\} (b_1\in \arg \max \{f(b) \,|\, b\in \mathbb{S}_1\}, \text{ resp.})\\
V_1 &=& \{u\in V\,|\, \tuple{u,v}=0\}\\
f_1 &=& f\upharpoonright V_1 \text{ be the restriction of } f \text{ to } V_1.
\end{eqnarray*}
If $(b_2,b_3,\dots,b_n)$ is an $f_1$-minimal ($f_1$-maximal, resp.) basis for $V_1$, then $(b_1,b_2,\dots,b_n)$ is an $f$-minimal ($f$-maximal, resp.) basis for $V$.
\end{itemize}
\end{definition}
\begin{definition}
We say that two orthonormal bases $(b_1,\dots,b_n)$ and $(e_1,\dots,e_n)$ of $V$ are \emph{equivalent} if for every $i\le n$ $b_i$ and $e_i$ are collinear, i.e. there is an element $\alpha_i\in \mathbb{F}$ with $|\alpha_i|=1$ such that $e_i=\alpha_i b_i$.
\end{definition}

\begin{theorem}\label{minimal}
For any norm $f:V\rightarrow \mathbb{R}^+$ and any $f$-minimal basis $(b_1,b_2,\dots,b_n)$ it holds that:
\begin{equation*}
(2^n-1)f(v)\ge \sum_{i=1}^n |\tuple{v,b_i}| f(b_i).
\end{equation*}
\end{theorem}
\begin{proof}
First note that for any $v\in V$ there are unique $\alpha\in \mathbb{F}$
and vector $u\in span(b_2,\dots,b_n)$ such that:
\begin{equation*}
v=u+\alpha b_1.
\end{equation*}
The statement being obvious for $v=0$, we assume that $v\neq 0$ and set $v'= \frac{v}{\|v\|}$. Then $v'\in \mathbb{S}_1$ and by the definition of $b_1$, we get that:
\begin{equation*}
f(v')\ge f(b_1)\ge \frac{|\tuple{v,b_1}|}{\|v\|}f(b_1)=\frac{|\alpha|}{\|v\|}f(b_1),
\end{equation*}
where we used the Cauchy-Schwartz inequality. So far we have that:
\begin{equation*}
f(v)=\|v\| f(v') \ge |\alpha| f(b_1).
\end{equation*}
This settles the case where $n=1$. Alternatively, i.e. for $n\ge 2$, we use the triangle inequality for $u=v - \alpha b_1$ to conclude that:
\begin{equation*}
f(u) = f(v-\alpha b_1) \le f(v) + |\alpha| f(b_1) \le f(v) + f(v)=2f(v).
\end{equation*}
With this inequality at hand we can conclude the proof the by induction on $n$. As we noticed, the case $n=1$ is settled. Assume that the conclusion of the lemma holds for any $(n-1)$-dimensional vector space $V'$ and consider the vector space $V$ of dimension $n$. Let $V'=span(b_2,\dots,b_n)$ and $f'=f\upharpoonright V'$. It follows, by definition, that $(b_2,\dots,b_n)$ is an $f'$-minimal basis for $V'$ and therefore by the induction hypothesis:
\begin{equation*}
(2^{n-1}-1)f'(u) \ge \sum_{i=2}^n |\tuple{u,b_i}| f(b_i) \text{ for any } u\in V'.
\end{equation*}
Finally, for any non-zero $v=u+\alpha b_1$ with $\alpha\in \mathbb{F}$ and $u\in V'$ we have proven that:
\begin{equation*}
f(v) \ge 2 f(u)=2f(u') \text{ and } f(v) \ge |\alpha| f(b_1).
\end{equation*}
Summing up, and using that $\tuple{v,b_i}=\tuple{u,b_i}$ for $i\ge 2$, we conclude that:
\begin{eqnarray*}
(2^n-1) f(v)&=& 2(2^{n-1}-1) f(v) + f(v) \\
&\ge& (2^{n-1}-1) f(u) + |\alpha|f(b_1)\\
&\ge& |\alpha| f(b_1) + \sum_{i=2}^n |\tuple{u,b_i}|f(b_i)\\
&=& |\tuple{v,b_1}|f(b_1) + \sum_{i=2}^n |\tuple{v,b_i}| f(b_i)\\
&=&\sum_{i=1}^n |\tuple{v,b_i}| f(b_i).
\end{eqnarray*}
\end{proof}

\begin{theorem}\label{maximal}
For any norm $f:V\rightarrow \mathbb{R}^+$ and any $f$-maximal basis $(b_1,b_2,\dots,b_n)$ it holds that:
\begin{equation*}
(2^n-1) f(v) \ge \sum_{i=1}^n |\tuple{v,b_i}| f(b_i).
\end{equation*}
\end{theorem}
\begin{proof}
We proceed by induction on $n=dim V$. For $n=1$ there is nothing to prove. Assume that the statement of the lemma holds for vector spaces with dimensionality $n$.
Let $f:V\rightarrow \mathbb{R}^+$ be a norm. Since $f(b_1)=\max_{v\in \mathbb{S}_1} f(v)$, by Lemma~\ref{basic} we have that:
\begin{equation*}
f(u)\ge |\tuple{u,b_1}|f(b_1).
\end{equation*}
Let $V_1=span(b_2,\dots,b_n)$ and consider an arbitrary vector $u=\alpha b_1 + v$ with $v\in V_1$. Then, we have:
\begin{equation*}
f(u) \ge |\alpha| f(b_1).
\end{equation*}
Furthermore, by the triangle inequality we have:
\begin{equation*}
f(u) +|\alpha|f(b_1)=f(\alpha b_1 + v) + f(-\alpha b_1)\ge f(v).
\end{equation*}
Summing up we obtain that $2f(u) \ge f(u) + |\alpha| f(b_1)\ge f(v)$. To complete the inductive step, we use that $\tuple{u,b_i}=\tuple{v,b_i}$ for $i\ge 2$ and the inductive hypothesis for $V_1$ and $f_1=f\upharpoonright V_1$. Thus, we compute:
\begin{eqnarray*}
(2^n-1) f(u) &\ge& f(u) + 2(2^{n-1}-1) f(u)\\
&\ge& |\alpha| f(b_1) + (2^{n-1}-1) f(v)\\
&\ge& |\alpha|f(b_1) +\sum_{i=2}^n |\tuple{v,b_i}|f(b_i)\\
&=&|\alpha|f(b_1) +\sum_{i=2}^n |\tuple{u,b_i}|f(b_i)\\
\end{eqnarray*}
where the second line follows by the fact $f(u)\ge |\alpha|f(b_1)$ and $2f(u)\ge f(v)$, and the third line follows by the inductive hypothesis.
\end{proof}
\section{Lower Bounds}\label{sec:lower_bounds}
In this section we prove that the results from Theorems~\ref{minimal} and~\ref{maximal} are sharp. Our constructions use the characterisation from Lemma~\ref{basic} and Remark~\ref{seminorm} as a basic tool to define norms. In both cases given an $\varepsilon>0$, we construct a finite set $U\subseteq \mathbb{R}^n$ that defines a norm:
\begin{equation*}
f:\mathbb{C}^n\rightarrow \mathbb{R}^+ \text{ s.t. } f(v) =\max_{u\in U} |\tuple{v,u}|.
\end{equation*}
The set $U$ is tailored in such a way that it admits:
\begin{enumerate}
\item as $f$-minimal ($f$-maximal, resp.) an orthonormal\footnote{Actually, it is unique up to equivalence.} basis $$(e_1,\dots,e_n)\subset \mathbb{R}^n$$ and
\item a witness $v_0\in \mathbb{R}^n$ such that:
\begin{equation*}
(2^n-1-\varepsilon) f(v_0)<\sum_{i=1}^n |\tuple{v_0,e_i}| f(e_i).
\end{equation*}
\end{enumerate}
Since $U\subseteq \mathbb{R}^n$, $(e_1,\dots,e_n)\subseteq \mathbb{R}^n$ and $v_0\in \mathbb{R}^n$, the restriction $f_R$ of $f$ to $\mathbb{R}^n$ provides a norm:
\begin{equation*}
f_R:\mathbb{R}^n\rightarrow \mathbb{R}^+ \text{ with } f_R(v) = f(v)
\end{equation*}
that admits as $f_R$-minimal ($f$-maximal, resp.) basis again $(e_1,\dots,e_n)$ and $v_0$ is a witnesses that:
\begin{equation*}
(2^n-1-\varepsilon) f_R(v_0)<\sum_{i=1}^n |\tuple{v_0,e_i}| f_R(e_i).
\end{equation*}

\begin{proposition}\label{min_asymptotic}
Let $n\ge 1$be an integer and $V=\mathbb{C}^n$. Then for any $\varepsilon>0$, there is a norm $f:V\rightarrow \mathbb{R}^+$ and a nonzero vector $v\in \mathbb{R}^n$ such that:
\begin{enumerate}
\item $f$ admits a unique, up to equivalence, $f$-minimal basis $$(e_1,\dots,e_n)\subseteq \mathbb{R}^n$$ and
\item
\begin{equation*}
(2^n-1-\varepsilon)f(v)\le \sum_{i=1}^n |\tuple{v,e_i}| f(e_i).
\end{equation*}
\end{enumerate}
\end{proposition}
\begin{proof}
Let $(e_1,e_2,\dots,e_n)\subseteq \mathbb{R}^n$ be a fixed orthonormal basis for $V$ and $\varepsilon>0$. We set $V_i=span(e_i,\dots,e_n)$. Finally, let $s\in (0,1)$ whose value will be specified appropriately later. We set out to inductively construct sets $U_i\subseteq V_i\cap \mathbb{R}^n$, norms $f_i:V_i\rightarrow \mathbb{R}^+$ and witnesses $v^{(i)}\in V_i\cap \mathbb{R}^n$ with the following properties:
\begin{enumerate}
\item $f_i(v)=\max \{|\tuple{v,u}|\,|\, u\in U_i\}$,
\item $f_i(e_i)=1=\min_{v\in \mathbb{S}_1\cap V_i} f_i(v)$,
\item $f_{i}(v)=\frac{1}{s^2(2s+1)}f_{i+1}(v)$ for all $v\in V_{i+1}$,
\item $\tuple{v^{(i)},e_j}=(2s^2)^{j-i}$ is such that $f_i(v^{(i)})\in [1,1+2s)$ and:
\begin{equation*}
(2^{n-i}-1)f(v^{(i)})<\sum_{j=i}^n |\tuple{v^{(i)},e_j}|(1+2s)^{2(j-i)+1} f_i(e_j))
\end{equation*}
\end{enumerate}
We start with $U_n=\{e_n\}$ and $f_n(\alpha e_n)=|\alpha|$, $v^{(n)}=e_n$. It is straightforward to see that these objects satisfy the above assumptions. Assume that for some $i>1$ the set $U_i$, the norm $f_i$ and the witness $v^{(i)}$ are defined and satisfy the above assumptions. We define $U_{i-1}$, $f_{i-1}$ and $v^{(i-1)}$ as follows. Let
\begin{equation*}
U_i^+ =\{u\in U_i \,|\, \tuple{v^{(i)},u}\ge 0\}\text{ and }U_i^- =\{u\in U_i \,|\, \tuple{v^{(i)},u}< 0\}.
\end{equation*}
Next we define:
\begin{eqnarray*}
U_{i-1} &= \left\{ e_{i-1} + \frac{1}{s(2s+1)}u\,|\, u\in U^+_i\right\} \cup \left\{ e_{i-1} - \frac{1}{s^2(2s+1)}u\,|\, u\in U^+_i\right\}\\
&\cup \left\{ e_{i-1} - \frac{1}{s(2s+1)}u\,|\, u\in U^-_i\right\} \cup \left\{ e_{i-1} + \frac{1}{s^2(2s+1)}u\,|\, u\in U^-_i\right\}.
\end{eqnarray*}
Now, we define $f_{i-1}$ and $v_{i-1}$ as:
\begin{eqnarray*}
f_{i-1} &=& \max\{|\tuple{v,u}|\, |\, u\in U_{i-1}\}\\
v^{(i-1)} &=& e_{i-1} + 2s^2 v^{(i)}.
\end{eqnarray*}
Note that since, by assumption, $U_i\subseteq \mathbb{R}^n$ and $v^{(i)}\in \mathbb{R}^n$, the sets $U_i^+$ and $U_i^{-}$ are well-defined. Hence, it should be clear that the $U_{i-1}\subseteq \mathbb{R}^n$ and $v^{(i-1)}\in \mathbb{R}^n$.

Now we verify that $U_{i-1}$, $f_{i-1}$ and $v^{(i-1)}$ possess the desired properties:
\begin{enumerate}
\item The first property is satisfied by definition.
\item The third property is also clear for $s\in (0,1)$.
\item To see that the second property holds, first note that:
\begin{equation*}
\tuple{e_{i-1},u}=1 \text{ for all } u\in U_{i-1}.
\end{equation*}
Next, consider an arbitrary element $v\in U_{i-1}\cap \mathbb{S}_1$. It has a unique representation $v=\alpha e_{i-1} + v'$ with $|\alpha|^2 + \|v'\|^2=1$ with $v'\in V_i$.
By the induction hypothesis we have that $f_i(v')\ge \|v'\|$. Since $U_i$ is finite, the value $f_i(v')=|\tuple{v',u'}|$ is attained for some $u'\in U_i$. Let us fix such $u'$ and set $a,b\in \mathbb{R}$ such that:
\begin{equation*}
\tuple{v',u'}= a+ ib.
\end{equation*}
We also set $c,d\in \mathbb{R}$ such that $\alpha=c+id$. Thus, for if $\sigma\in \{s^{-1},s^{-2},-s^{-1},-s^{-2}\}$ we have:
\begin{equation*}
\tuple{v,e_{i-1}+\sigma u'} = c+id +\sigma(a+ib)=(c+\sigma a) + i (d+ \sigma b).
\end{equation*}
Consequently:
\begin{eqnarray*}
|\tuple{v,e_{i-1}+\sigma u'}| &=& (c+\sigma a)^2 + (d+\sigma b)^2\\
&=& c^2 + d^2 + \sigma^2 (a^2+b^2) + 2\sigma(ac + bd) \\
&=& |\alpha|^2 + \sigma^2|\tuple{v',u'}|^2 + 2\sigma(ac + bd)\\
&=& |\alpha|^2+ \sigma^2 f^2_i(v') + 2\sigma(ac+bd) \\
&\ge & |\alpha|^2 + \sigma^2 \|v'\|^2 + 2\sigma(ac+bd).
\end{eqnarray*}
By construction, we can always choose the sign of $\sigma$, i.e. we have either the option $\sigma\in \{s^{-1},-s^{-2}\}$ or  $\sigma\in \{-s^{-1},s^{-2}\}$. Now, choosing $\sigma$ such that the sign of $\sigma$ is the same as $(ac+bd)$ we conclude that:
\begin{eqnarray*}
|\tuple{v,e_{i-1}+\sigma u'}| &\ge& |\alpha|^2 + \sigma^2 \|v'\|^2 + 2\sigma(ac+bd)\\
&\ge& |\alpha|^2 +\sigma^2 \|v\|^2\\
&\ge& |\alpha|^2 + \|v'\|^2\\
&=& 1,
\end{eqnarray*}
where the first inequality follows by the choice of $\sigma$ and the second by the fact that $|\sigma|>1$. Furthermore the last inequality turns into equality if and only if $v'=0$.
This proves that for any vector $v\in \mathbb{S}_1\cap V_{i-1}$ which is not collinear with $e_{i-1}$, it holds that $f_i(v)>1$. Consequently:
\begin{equation*}
f_{i-1}(e_{i-1})=\min_{v\in V_{i-1}\cap \mathbb{S}_1} f_{i-1}(v)
\end{equation*}
and the minimum is attained only for vectors of the form $\alpha e_{i-1}$ with $|\alpha|=1$. Hence w.l.o.g. $e_1$ belongs to the minimal basis.
\item Finally, we check the last condition. Let $v^{(i-1)}\in  V_{i-1}$ be such that:
\begin{equation*}
v^{(i-1)}=e_{i-1} + 2s^2 v^{(i)}.
\end{equation*}
First, let $u\in U_i^+$. In particular, $\tuple{u,v^{(i)}}\ge 0$. Therefore:
\begin{eqnarray*}
\tuple{v^{(i-1)},e_{i-1}-\frac{1}{s^2(2s+1)}u}&=&\tuple{e_{i-1} + 2s^2v^{(i)},e_{i-1}-\frac{1}{s^2(2s+1)}u}\\
&=& 1- \frac{2}{2s+1}\tuple{v^{(i)},u'}\\
&=&1- \frac{2}{2s+1}|\tuple{v^{(i)},u}|\le 1.
\end{eqnarray*}
Since $u\in U_i$, we have that $f_i(v^{(i)})\ge |\tuple{v^{(i)},u}|$ and by the assumption that $f_i(v^{(i)})\le 1+2s$ we conclude that:
\begin{equation*}
\tuple{v^{(i-1)},e_{i-1}-\frac{1}{s^2(2s+1)}u}=1- \frac{2}{2s+1}|\tuple{v^{(i)},u}|>1-2=-1.
\end{equation*}
Hence $|\tuple{v^{(i-1)},e_{i-1}-\frac{1}{s^2(2s+1)}u}|\le 1$ for all $u\in U_i^+$. On the other hand:
\begin{eqnarray*}
\tuple{v^{(i-1)},e_{i-1}+\frac{1}{s(2s+1)}u} &=&1+ \frac{2s^2}{s(2s+1)}\tuple{v^{(i)},u}\\
&=& 1+ \frac{2s}{1+2s}|\tuple{v^{(i)},u}|\in [1,1+2s).
\end{eqnarray*}
Similarly, for $u\in U_i^-$ we have that:
\begin{eqnarray*}
\tuple{v^{(i-1)},e_{i-1}+\frac{1}{s^2(2s+1)}u}&=&1- \frac{2}{2s+1}|\tuple{v^{(i)},u}|\in (-1,1]\\
\tuple{v^{(i-1)},e_{i-1}-\frac{1}{s(2s+1)}u}&=&1+ \frac{2s}{2s+1}|\tuple{v^{(i)},u}|\in [1,1+2s).
\end{eqnarray*}
Hence the maximum value of $\tuple{v^{(i-1)},u'}$ when $u'\in U_{i-1}$ is attained for some $u'$ of the form
\begin{eqnarray*}
u'&=&e_{i-1}+\frac{1}{s(2s+1)}u \text{ with } u\in U_i^+ \text{ or}\\
u'&=&e_{i-1}-\frac{1}{s(2s+1)}u \text{ with } u\in U_i^-
\end{eqnarray*}
such that $|\tuple{v^{(i)},u}|$ is maximised. This shows that $$f_{i-1}(v^{(i-1)})=1+\frac{2s}{2s+1}f_{i}(v^{(i)}).$$
So far we have that $f_{i-1}(v^{(i-1)})\in [1,1+2s)$.

We proceed to show that the last inequality holds, To this end, first note:
\begin{eqnarray*}
\tuple{v^{(i-1)},e_{i-1}} &=&1\\
\tuple{v^{(i-1)},e_{j}} &=& 2s^2 \tuple{v^{(i)},e_j} \text{ for } j\ge i.
\end{eqnarray*}
Recalling that $f_{i-1}(e_{i-1})=1$ and $f_{i-1}(e_j)=s^2(2s+1)f_i(e_j)$ for $j\ge i$, we conclude that:
\begin{equation*}
(1+2s)|\tuple{v^{(i-1)},e_{j}}|f_{i-1}(e_j)=2|\tuple{v^{(i)},e_j}| f_i(e_j) \text{ for } j\ge i.
\end{equation*}
Therefore, using that $f_{i-1}(v^{(i-1)})< 1+2s\le (1+2s)f_i(v^{(i)})$, we compute:
\begin{eqnarray*}
(2^{n-i+1}-1) f_{i-1}(v^{(i-1)}) &=& f_{i-1}(v^{(i-1)})+ 2(2^{n-i}-1)f_{i-1}(v^{(i-1)})  \\
&< & 1+2s + 2(2^{n-i}-1)f_{i-1}(v^{(i-1)})\\ 
  &\le &(1+2s)(1+2f_i(v^{(i)})(2^{n-i}-1))\\ 
\text{(inductive hypothesis) } &\le & (1+2s) \left(1 + \right .\\
&&\left . 2\sum_{j=i}^n |\tuple{v^{(i)},e_j}|(1+2s)^{2(j-i)+1} f_i(e_j)\right)\\
&\le& (1+2s)( |\tuple{v^{(i-1)},e_{i-1}}| f(e_{i-1})+ \\
&&\sum_{j=i}^n (1+2s)^{2(j-i)+3}|\tuple{v^{(i-1)},e_j}|f_{i-1}(e_j))\\
(\text{ setting }i'=i-1) &=&\sum_{j=i'}^n(1+2s)^{2(j-i')+1}|\tuple{v^{(i')},e_j}|f_{i'}(e_j)
\end{eqnarray*}
as required.
\end{enumerate}
Now letting $s$ tend to zero, we see that $f_1(v^{(1)})$ satisfies the conclusion of the lemma.
\end{proof}

\begin{proposition}\label{max_asymptotic}
Let $n\ge 1$ be an integer and $V=\mathbb{C}^n$. For any real number $\varepsilon>0$, there is a norm $f:V\rightarrow \mathbb{R}^+$ and a non-zero vector $v\in \mathbb{R}^n$ such that:
\begin{enumerate}
\item $f$ admits a unique up to equivalence $f$-maximal basis $$(e_1,e_2,\dots,e_n)\subseteq \mathbb{R}^n,$$
\item
\begin{equation*}
(2^n-1-\varepsilon) f(v)<\sum_{i=1}^n |\tuple{v,e_i}|f(e_i).
\end{equation*}
\end{enumerate}
\end{proposition}
\begin{proof}
Let $\varepsilon>0$ be fixed and $e_1,\dots,e_n\in \mathbb{R}^n$ be an orthonormal basis of $V$. We are going to construct a norm $f:V\rightarrow \mathbb{R}^+$ whose unique $f$-maximal basis is $(e_1,\dots,e_n)$ and satisfies the conclusion of the lemma. To this end consider a real number $c\in (0,1)$ and an angle $\alpha\in (0,\pi)$ whose price values will be determined appropriately.


Given, the constants $c$ and $\alpha$ we define the vector $u'_k\in V$ for $k=1,2\dots,n$ as follows:
\begin{eqnarray*}
u'_1&=& e_1\\
u'_{k+1}&=& \sum_{j=1}^k e_j \sin^{j-1}\alpha \cos \alpha + e_{k+1}\sin^k \alpha.
\end{eqnarray*}
We  set $u_k=c^{k-1}u_k'$ and define the norm $f:V\rightarrow \mathbb{R}^+$ as:
\begin{equation*}
f(v) =\sup\{|\tuple{v,u_k}| \,|\, k\le n\}.
\end{equation*}
By Remark~\ref{seminorm} we know that $f$ is a semi-norm. Since $f(e_k)\ge |\tuple{e_k,u_k}|>0$, we see, again by Remark~\ref{seminorm}, that $f$ is a norm.
Further, by Lemma~\ref{basic}, we know that the maximum value $f(v)$ on $\mathbb{S}_1$ is attained at a vector that is collinear to some of the vectors $u_k$ and is equal to $f(v)=\|u_k\|$. Since $\|u_k\|=c_k$, we conclude that the first vector of the maximal basis is $e_1=u_1$. Next, the subspace, $V_1$, of $V$ that is orthogonal to $e_1$ is spanned by $(e_2,\dots,e_n)$ and therefore $f_1=f\upharpoonright V_1$ is actually:
\begin{equation*}
f_1(v) = \sup\{|\tuple{v,u_k-u_1\cos \alpha}\,|\, k\le n\}.
\end{equation*}
Since $\|u_k-u_1\cos\alpha|=c^{k-1}\sin\alpha$, applying again Lemma~\ref{basic}, we conclude that the maximal value of $f_1$ on $\mathbb{S}_1$ is attained at $e_2$. Proceeding inductively, we may prove that $(e_1,\dots,e_n)$ is the unique, up to equivalence, $f$-maximal basis. Note that:
\begin{equation*}
f(e_i)=\tuple{e_i,u_i}=c^{i-1} \sin^{i-1} \alpha.
\end{equation*}

Let $w_1>0$ and $w_i=w_1\tg^{i-1}\frac{\alpha}{2}$. It should be clear that:
\begin{eqnarray*}
w_i \cos{\alpha}+w_{i+1} \sin{\alpha}&=& w_i (\cos{\alpha} +\sin{\alpha}\tg\frac{\alpha}{2})\\
&=& w_i(2\cos^2{\alpha/2}-1 + 2\sin^2{\alpha/2})=w_i.
\end{eqnarray*}
Therefore, setting $w=\sum_{i=1}^n w_i e_i$, we obtain that $$\tuple{w,u_k}=c^{k-1}\tuple{w,u_k'}=c^{k-1}w_1$$ and thus $|\tuple{w,u_k}|=c^{k-1}w_1\le w_1$ with equality if and only if $k=1$. Hence $f(w)=w_1$.

On the other hand:
\begin{eqnarray*}
w_i f(e_i) &=& c^{i-1}w_1 \tg^{i-1} \frac{\alpha}{2}\sin^{i-1}{\alpha}\\
&=& c^{i-1}w_1 \left(2\sin^2\frac{\alpha}{2}\right)^{i-1}\\
&=& 2^{i-1} c^{i-1}w_1\sin^{(2(i-1))}\frac{\alpha}{2}.
\end{eqnarray*}
It follows that:
\begin{equation*}
\sum_{i=1}^n w_i f(e_i)\ge w_1 \sum_{i=1}^n 2^{i-1} [c^{i-1} \sin^{2(i-1)}\frac{\alpha}{2}].
\end{equation*}
Clearly, letting $c$ tend to $1$ and $\alpha$ tend to $\pi-0$, the right hand side tends to $w_1\sum_{i=1}^{n}2^{i-1}=(2^n-1)w_1=(2^n-1)f(w)$. Thus, for any $\varepsilon>0$, we can find appropriate $c\in (0,1)$ and $\alpha\in (0,\pi)$ such that $(2^{n}-1-\varepsilon)f(w)<\sum_{i=1}^n w_i f(e_i)$.
\end{proof}

\section{Equivalence of Bases}\label{sec:comparison}
\begin{definition}
Let $f:V\rightarrow \mathbb{R}^+$ be a norm. Let $(e_1,e_2,\dots,e_n)$ be an orthonormal basis for $V$ arranged in increasing order w.r.t. $f$, i.e.:
\begin{equation*}
f(e_1)\le f(e_2)\le \dots\le f(e_n).
\end{equation*}
\begin{enumerate}
\item For a constant $c\in \mathbb{R}^+$, we say that $(e_1,e_2,\dots,e_n)$ satisfies the property $P_f(c)$ if for every $\beta_1,\beta_2,\dots,\beta_n\in \mathbb{F}$ it holds:
\begin{equation*}
c f\left(\sum_{i=1}^n \beta_i e_i\right)\ge \sum_{i=1}^n |\beta_i|f(e_i).
\end{equation*}
\item More generally, for constants $c_1,c_2,\dots,c_n\in \mathbb{R}^+$, we say that $(e_1,\dots,e_n)$ satisfies the property $HP_f(c_1,c_2,\dots,c_n)$ if for every $i$ and every $\beta_i,\beta_{i+1},\dots,\beta_n$ it holds that:
\begin{equation*}
c_i f\left(\sum_{j=i}^n \beta_i e_i\right)\ge \sum_{j=i}^n |\beta_i|f(e_i).
\end{equation*}
\end{enumerate}
\end{definition}
\begin{remark}
With these notions, Lemma~\ref{maximal} states that there is an $f$-maximal basis that satisfies $P_f(2^n-1)$. Furthermore, since every prefix of an $f$-maximal basis, $(b_1,b_2,\dots,b_n)$, is a maximal basis for its linear span, it follows that the $f$-maximal basis from Lemma~\ref{maximal} actually satisfies $HP_f(2^n-1,2^{n-1}-1,\dots,1)$.

On the other hand, Lemma~\ref{minimal} states that every $f$-minimal basis satisfies $P_f(2^n-1)$.

It is also obvious that if a basis $(e_1,\dots,e_n)$ satisfies $P_f(c)$, then it satisfies $HP_f(c,c,\dots,c)$. Conversely, if an orthonormal basis $(e_1,\dots,e_n)$ satisfies $HP_f(c_1,\dots,c_n)$, then it satisfies $P_f(c_1)$.
\end{remark}
\begin{lemma}\label{bases_equivalence}
Let $(b_1,b_2,\dots,b_n)$ and $(e_1,e_2,\dots,e_n)$ be orthonormal bases on $V$ and $f:V\rightarrow \mathbb{R}^+$ be a norm such that:
\begin{eqnarray*}
f(b_1)\le f(b_2)\le \dots \le f(b_n)&\text{ and } \\
f(e_1)\le f(e_2)\le \dots \le f(e_n).&
\end{eqnarray*}
Then:
\begin{enumerate}
\item if $(b_1,\dots,b_n)$ is $f$-minimal, then for every $i\le n$, $ f(b_i)\le \sqrt{i} f(e_i)$,
\item if $(e_1,e_2,\dots,e_n)$ satisfies $HP_f(c_1,c_2,\dots,c_n)$, then for every $i\le n$ it holds that $f(e_i)\le c_i\sqrt{i} f(b_i)$.
\end{enumerate}
\end{lemma}
\begin{proof}
Let $i\le n$ and note that $(b_1,b_2,\dots,b_{i-1})$ spans an $i$-dimensional subspace of $V$, whereas $(e_1,e_2,\dots,e_{i})$ spans an $i$-dimensional subspace of $V$. Hence there is a non-zero vector $b'\in span(e_1,\dots,e_i)$ that is orthogonal to all the vectors $b_1,b_2,\dots,b_{i-1}$. Indeed, it is straightforward that the system:
\begin{equation*}
\sum_{j=1}^i \alpha_j \tuple{e_j,b_k}= 0 \text{ for } k\le i-1
\end{equation*}
is overdetermined. Hence it admits a non-zero solution $\alpha=(\alpha_1,\dots,\alpha_i)$ and since $e_1,\dots,e_i$ are independent, $b'=\sum_{j=1}^i \alpha_j e_j$ is non-zero. Next, without loss of generality, we may and we do assume that $\|b'\|=1$. Thus $b'\in \mathbb{S}_1$ and $b'$ is orthogonal to all the vectors $b_1,\dots,b_{i-1}$. By the definition of $b_i$, it follows that:
\begin{equation*}
f(b_i)\le f(b')=f\left(\sum_{j=1}^i \alpha_j e_j\right)\le \sum_{j=1}^i |\alpha_j| f(e_j),
\end{equation*}
where the last inequality follows by the triangle inequality. Finally, by the arrangement of the vectors $e_1,\dots,e_i$ we have $f(e_j)\le f(e_i)$ for every $j\le i$ and by the Cauchy-Schwartz inequality we have $\sum_{j=1}^i |\alpha_j|\le \sqrt{i}\sqrt{\sum_{j=1}^i |\alpha_j|^2}=\sqrt{i}$. Summing up we obtain:
\begin{equation*}
f(b_i)\le \sum_{j=1}^i |\alpha_j| f(e_j)\le \sum_{j=1}^i |\alpha_j| f(e_i)\le \sqrt{i} f(e_i).
\end{equation*}

For the second part of the statement, we proceed similarly. Let $i\le n$. Then $(b_1,\dots,b_i)$ span a linear space of dimensionality $i$ whereas $(e_1,\dots,e_{i-1})$ span a linear space of dimensionality $i-1$. Then, as above, there is a non-zero vector $v\in span(b_1,\dots,b_i)$ that is orthogonal to all the vectors $e_1,\dots,e_{i-1}$. Without loss of generality we may and we do assume that $v$ is a unit vector. Since $v$ is orthogonal to $e_1,\dots,e_{i-1}$, it belongs to the linear space spanned by $(e_i,\dots,e_n)$. Hence $v$ can be written as $v=\sum_{j=i}^n \alpha_j e_j$. By the $HP_f(c_1,\dots,c_n)$ of the basis $(e_1,\dots,e_n)$, we conclude that:
\begin{equation*}
c_i f(v) = c_if\left(\sum_{j=i}^n \alpha_j e_j\right)\ge  \sum_{j=i}^n |\alpha_j| f(e_j)\ge  \sum_{j=i}^n |\alpha_j| f(e_i)\ge f(e_i),
\end{equation*}
where the last but one inequality follows by the ordering of the vectors $(e_1,\dots,e_n)$ and the last inequality follows by the fact that $\sum_{j=i}^n|\alpha_j|^2=1$, as $v$ is a unit vector.

On the other hand, $v\in span(b_1,\dots,b_i)$ and hence $v=\sum_{j=1}^i \beta_j b_j$. Since $\|v\|=1$, we have that $\sum_{j=1}^i |\beta_j|^2=1$. Therefore, applying the triangle inequality, we obtain:
\begin{equation*}
f(v)=f\left(\sum_{j=1}^i \beta_j b_j\right)\le \sum_{j=1}^i |\beta_j| f(b_j)\le  \sum_{j=1}^i |\beta_j|f(b_i) \le  \sqrt{i} f(b_i),
\end{equation*}
where the last but one inequality follows by the order of $(b_1,\dots,b_n)$ and the last one is a trivial application of the Cauchy-Schwartz inequality.

Summing up we get:
\begin{equation*}
f(e_i) \le c_i f(v) \le c_i \sqrt{i} f(b_i)
\end{equation*}
as claimed.
\end{proof}
\begin{corollary}
Let $f:V\rightarrow \mathbb{R}^+$ be a norm and $(b_1,b_2,\dots,b_n)$ and $(e_1,e_2,\dots,e_n)$ be orthonormal bases  such that:
\begin{eqnarray*}
f(b_1)\le f(b_2)\le \dots \le f(b_n)&\text{ and } \\
f(e_1)\le f(e_2)\le \dots \le f(e_n).&
\end{eqnarray*}
\begin{enumerate}
\item If $(b_1,\dots,b_n)$ and $(e_1,\dots,e_n)$ are $f$-minimal, then $$\frac{1}{\sqrt{i}}\le \frac{f(b_i)}{f(e_i)}\le \sqrt{i}.$$
\item If $(b_1,\dots,b_n)$ is $f$-minimal and  $f$-maximal basis $(e_1,\dots,e_n)$, then $\frac{1}{\sqrt{i}(2^{n-i+1}-1)}\le \frac{f(b_i)}{f(e_i)}\le \sqrt{i}$.
\item If $(b_1,\dots,b_n)$ and $(e_1,\dots,e_n)$ are $f$-maximal, then $$\frac{1}{\sqrt{i}(2^{n-i+1}-1)}\le \frac{f(b_i)}{f(e_i)}\le \sqrt{i}(2^{n-i+1}-1).$$
\end{enumerate}
\end{corollary}
Under certain additional assumptions, Lemma~\ref{bases_equivalence} can be inverted as follows:
\begin{proposition}
Assume that $f:V\rightarrow\mathbb{R}^+$ is a norm and $(b_1,\dots,b_n)$ and $(e_1,\dots,e_n)$ and $c\ge 1$ are orthonormal bases for $V$ such that:
\begin{eqnarray*}
f(b_1)\le f(b_2)\le \dots \le f(b_n) \\
f(e_1)\le f(e_2)\le \dots \le f(e_n) \\
\forall i (f(e_i)/c\le f(b_i)\le c f(e_i)).
\end{eqnarray*}
If $(b_1,\dots,b_n)$ satisfies $H_f(c_1)$ and additionally there is $\alpha\ge 1$ such that:
\begin{equation*}
\forall i\neq j \exists \alpha_{i,j}(|\alpha_{i,j}|<\alpha \text{ and } \tuple{e_i-b_i,e_j-\alpha_{i,j} b_j}=0),
\end{equation*}
then $(e_1,\dots,e_n)$ satisfies $H_f(\alpha c^2 c_1)$.
\end{proposition}
\begin{proof}
Since $(b_1,\dots,b_n)$ satisfies $H_f(c_1)$, it follows that:
\begin{equation*}
c_1 f(e_i)\ge \sum_{j=1}^n |\tuple{e_i,b_j}| f(b_j).
\end{equation*}
Hence:
\begin{eqnarray*}
f(b_i) &\ge & f(e_i)/c \\
&\ge & c_1 f(e_i)/(cc_1)\\
&\ge & \sum_{j=1}^n |\tuple{e_i,b_j}| f(b_j)/c\\
&\ge & \frac{1}{c^2}\sum_{j=1}^n |\tuple{e_i,b_j}| f(e_j).
\end{eqnarray*}
Note that the condition $\tuple{e_i-b_i,e_j-\alpha_{i,j} b_j}=0$ can be rewritten as:
\begin{equation*}
\tuple{e_i,e_j} + \overline{\alpha_{i,j}}\tuple{b_i,b_j} - \tuple{b_i,e_j} - \overline{\alpha_{i,j}}\tuple{e_i,b_j}=0.
\end{equation*}
For $i\neq j$ it holds that $\tuple{b_i,b_j}=\tuple{e_i,e_j}=0$ and therefore for $i\neq j$ we have:
\begin{equation*}
|\tuple{b_i,e_j}|= |\alpha_{i,j}| |\tuple{e_i,b_j}|\le \alpha |\tuple{e_i,b_j}|.
\end{equation*}
Hence the above inequality implies that:
\begin{equation*}
f(b_i) \ge \frac{1}{c^2}\sum_{j=1}^n |\tuple{e_i,b_j}| f(e_j) \ge \frac{1}{\alpha c^2}\sum_{j=1}^n |\tuple{b_i,e_j}| f(e_j).
\end{equation*}
Therefore:
\begin{equation*}
\alpha c^2 f(b_i) \ge \sum_{j=1}^n |\tuple{b_i,e_j}| f(e_j).
\end{equation*}
Finally, for arbitrary $v$ the above inequality and the validity of $H_f(c_1)$ for the basis $(b_1,\dots,b_n)$ imply:
\begin{eqnarray*}
\alpha c^2 c_1 f(v) &\ge& \alpha c^2 \sum_{i=1}^n |\tuple{v,b_i}| f(b_i)\\
&\ge & \sum_{i=1}^n\sum_{j=1}^n |\tuple{v,b_i}||\tuple{b_i,e_j}| f(e_j)\\
&= & \sum_{j=1}^n f(e_j) \sum_{i=1}^n |\tuple{v,b_i}\tuple{b_i,e_j}|\\
&\ge &\sum_{j=1}^n f(e_j)\left| \sum_{i=1}^n \tuple{v,b_i}\tuple{b_i,e_j} \right|\\
&=&\sum_{j=1}^n |\tuple{v,e_j}| f(e_j).
\end{eqnarray*}
This proves that $(e_1,\dots,e_n)$ satisfies $H_f(\alpha c^2 c_1)$.

\end{proof}

\section{Open Problems}\label{sec:open}
The definitions of $f$-minimal and $f$-maximal bases of a norm suggest a simple greedy strategy to find an orthonormal basis $(e_1,\dots,e_n)$ which satisfies the inequality:
\begin{equation*}
f(\sum_{i=1}^n \alpha_i e_i)\ge \frac{1}{2^n-1}\sum_{i=1}^n |\alpha_i|f(e_i).
\end{equation*}
Of course, the feasibility of this approach depends on the possibility to efficiently solve the optimisation problem:
\begin{equation*}
\max_{v\in \mathbb{S}_1} f(v) \text{ or } \min_{v\in \mathbb{S}_1} f(v).
\end{equation*}
Yet, since every (semi)norm $f$ is convex and $\mathbb{S}_1$ is compact, especially the second problem is well studied and efficient methods for its solution stay at hand.

However, the problem with the greedy approach is that the constant $2^n-1$ grows exponentially with $n$ and it may be inconvenient to prove precise bounds in general.
As we have proven in Lemma~\ref{min_asymptotic} and Lemma~\ref{max_asymptotic}, the constant $2^n-1$ cannot be improved under the suggested greedy strategy.
Thus, the natural question that arises is how this constant can be improved while preserving the clear structure of the bases that it implies. In this respect, we consider the following theoretical problems.

First, for a natural number $n\ge 1$, and a linear vector space $V$ with inner product, where $V=\mathbb{C}^n$ or $V=\mathbb{R}^n$. We define $c_n^{\perp}:=c_n^{\perp}(V)$ to be the least real number such that for every norm $f:V\rightarrow \mathbb{R}^+$ there is an orthonormal basis $(b_1,\dots,b_n)$ such that:
\begin{equation*}
f(\sum_{i=1}^n \alpha_i b_i)\ge \frac{1}{c_n^{\perp}}\sum_{i=1}^n |\alpha_i|f(b_i) \text{ for all } \alpha_1,\dots,\alpha_n \in \mathbb{F}.
\end{equation*}
It is known that for $V=\mathbb{R}^2$, $c_2^{\perp}=2$,~\cite{Lassak02}. The construction in~\cite{Lassak02} relies upon defining appropriate areas and the continuity principle to show existence. Is there a more explicit way to define such a basis? To the best knowledge of the authors, the techniques from $n=2$ do not extend to higher dimensions.

Secondly, for a natural number $n\ge 1$, and a linear vector space $V$ with inner product, where $V=\mathbb{C}^n$ or $V=\mathbb{R}^n$. We define $c_n^{\angle}=c_n^{\angle}(V)$ to be the least real number such that for every norm $f:V\rightarrow \mathbb{R}^+$ there is an basis $(b_1,\dots,b_n)$ of unit vectors such that:
\begin{equation*}
f(\sum_{i=1}^n \alpha_i b_i)\ge \frac{1}{c_n^{\angle}}\sum_{i=1}^n |\alpha_i|f(b_i) \text{ for all } \alpha_1,\dots,\alpha_n \in \mathbb{F}.
\end{equation*}
It is known that for $V=\mathbb{R}^2$, $c_2^{\angle}=\frac{3}{2}$,~\cite{Asplund60}.
This question is tightly related with John's Theorem~\cite{John48} which relies on volumes' optimisation.

Both questions can be uniformly stated as follows. Let $n\ge 1$, and $V=\mathbb{C}^n$ or $V=\mathbb{R}^n$ and $\alpha\in [0,1]$. Define $c_n^{\alpha}:=c_n^{\alpha}(V)$ to be the least real number such that for every norm $f:V\rightarrow \mathbb{R}^+$ there is a basis $(b_1,\dots,b_n)$ of unit vectors such that:
\begin{eqnarray*}
f\left(\sum_{i=1}^n \alpha_i b_i\right) &\ge &\frac{1}{c_n^{\alpha}}\sum_{i=1}^n |\alpha_i|f(b_i) \text{ for all } \alpha_1,\dots,\alpha_n \in \mathbb{F}\\
\text{subject to}: |\tuple{b_i,b_j}| &\le& \alpha\text{ for all } i\neq j.
\end{eqnarray*}
In this framework, $c_n^{\perp}(V)=c_n^{0}(V)$ and $c_n^{\angle}(V)=c_n^{1}(V)$. We consider that the freedom to vary $\alpha$ may be useful in applications where this kind of inequalities are to be combined with other classical inequalities where the scalar products of the basis' vectors has to be controlled.


\begin{thebibliography}{10}

\bibitem{Asplund60}
Edgar Asplund.
\newblock Comparison between {P}lane {S}ymmetric {C}onvex {B}odies and
  {P}arallelograms.
\newblock {\em Marh. Scand.}, 8:171--180, 1960.

\bibitem{CD14}
Philip Charpentier and Yves Dupain.
\newblock Extremal {B}ases, {G}eometrically {S}eparated {D}omains and
  {A}pplications.
\newblock {\em Algebra and Analysis}, 26(1):196--269, 2014.

\bibitem{H02}
Tosten Hefer.
\newblock H\"older and $l^p$ {E}stimates for $\overline{\partial}$ on {C}onvex
  {D}omains of {F}inite {T}ype {D}epending on {C}atlin's {M}ultiptype.
\newblock {\em Mathematische Zeitschrift}, 242:367--398, 2002.

\bibitem{H04}
Tosten Hefer.
\newblock {E}xtremal {B}ases and {H}\"older {E}stimates for
  $\overline{\partial}$ on {C}onvex {D}omains of {F}inite {T}ype.
\newblock {\em Michigan Math. J.}, 52:573--602, 2004.

\bibitem{John48}
Fritz John.
\newblock Extremum {P}roblems with {I}nequalities as {S}ubsidary {C}onditions.
\newblock In {\em In {S}tudies and {E}ssays presented to {R}. {C}ourant on his
  60th birthday,}, pages 187--204. Interscience {P}ublishers {I}nc., 1948.

\bibitem{Lassak02}
Marek Lassak.
\newblock Approximation of {C}onvex {B}odies {B}y {A}xially {S}ymmetric
  {B}odies.
\newblock {\em Proceedings of the American Mathematical Society},
  130(10):3075--3084, 2002.

\bibitem{M92}
Jeffery~D. McNeal.
\newblock Convex {D}omains of {F}inite {T}ypes.
\newblock {\em Journal of Functional Analysis}, 108:361--373, 1992.

\bibitem{M94}
Jeffery~D. McNeal.
\newblock Estimates on the {B}ergman {K}ernels of {C}onvex {D}omains.
\newblock {\em Advances in Mathematics}, 109:108--139, 1994.

\bibitem{NP03}
Nikolai Nikolov and Peter Pflug.
\newblock Estimates for the {B}ergman {K}ernel and {M}etric of {C}onvex
  {D}omains in $\mathbb{C}^n$.
\newblock {\em Annales Polonici Mathematici}, 81(1):73--78, 2003.

\bibitem{NPT13}
Nikolai Nikolov, Peter Pflug, and Pascal~J. Thomas.
\newblock On {D}ifferent {E}xtremal {B}ases for $\mathbb{C}$-convex {D}omains.
\newblock {\em Proceedings of the American Mathematical Society},
  141(9):3223--3230, 2013.

\bibitem{NPZ11}
Nikolai Nikolov, Peter Pflug, and Wlodzimierz Zwonek.
\newblock Estimates for {I}nvariant {M}etrics on $\mathbb{C}$-convex {D}omains.
\newblock {\em Transactions of the American Mathematical Society},
  363(12):6245--6256, 2011.

\bibitem{W22}
Hongyu Wang.
\newblock Estimates of the {K}obayashi {M}etric and {G}romov {H}yperbolicity on
  {C}onvex {D}omains of {F}inite {T}ype, 2022.

\end{thebibliography}

\end{document}